\documentclass[11pt,a4paper]{article}

\usepackage{amsthm}
\usepackage{tikz}
\usepackage{fullpage}
\usepackage{float}
\usepackage{multirow}
\usepackage{amsmath}
\usepackage{amssymb}
\usepackage{comment}
\usepackage{amsthm}
\usepackage{amscd}
\usepackage{amsfonts}
\usepackage{soul}
\usepackage{authblk}
\usepackage{subcaption}
\usepackage{adjustbox}

\usepackage{hyperref}
\usepackage{url}
\usepackage{mleftright} %%matrix with line
\usepackage{float}
\usepackage{multirow}
\usepackage{hyperref}
\PassOptionsToPackage{lowtilde}{url}

\newcommand\blfootnote[1]{%
	\begingroup
	\renewcommand\thefootnote{}\footnote{#1}%
	\addtocounter{footnote}{-1}%
	\endgroup
}

\newtheorem{theorem}{Theorem}[section]
\newtheorem{lemma}[theorem]{Lemma}
\newtheorem{corollary}[theorem]{Corollary}
\newtheorem{proposition}[theorem]{Proposition}

\newtheorem{problem}[theorem]{Problem}

\newtheorem{remark}[theorem]{Remark}

\def\ev{\mathop{\rm ev}\nolimits}

\def\1{{\bf 1}}

\def\b{\mbox{\boldmath $b$}}
\def\c{\mbox{\boldmath $c$}}

\def\m{\mbox{\boldmath $m$}}

\def\x{\mbox{\boldmath $x$}}

\DeclareMathOperator{\spec}{sp}
\def\vece{\mbox{\boldmath $e$}}
\def\EE{\mbox{\boldmath $E$}}
\def\Real{\mathbb R}

\usepackage[textwidth=2.5cm]{todonotes}

\usepackage{notation}

\begin{document}
	
	\title{On inertia and ratio type bounds for the $k$-independence number of a graph and their relationship}
	
	\author[1]{Aida Abiad}
	\author[2]{Cristina Dalf\'o}
	\author[3]{Miquel \`Angel Fiol}
	\author[4]{Sjanne Zeijlemaker}

	\affil[1]{\small{Dept. of Mathematics and Computer Science, Eindhoven University of Technology, The Netherlands\linebreak
			Dept. of Mathematics: Analysis, Logic and Discrete Mathematics, Ghent University, Belgium\linebreak
			Dept. of Mathematics and Data Science, Vrije Universiteit Brussel, Belgium\linebreak
			(\texttt{a.abiad.monge@tue.nl})}}
	\affil[2]{\small{Dept. de Matem\`atica, Universitat de Lleida,  Catalonia\linebreak
			(\texttt{cristina.dalfo@udl.cat})}}
	\affil[3]{\small{Dept. de Matem\`atiques, Universitat Polit\`ecnica de Catalunya, Catalonia,\linebreak  
			Barcelona Graduate School of Mathematics, Catalonia\linebreak
			Institut de Matem\`atiques de la UPC-BarcelonaTech (IMTech)\linebreak
			\texttt{(miguel.angel.fiol@upc.edu)}}}
	\affil[4]{\small{Dept. of Mathematics and Computer Science, Eindhoven University of Technology, The Netherlands\linebreak (\texttt{s.zeijlemaker@tue.nl}).}}
	
	\date{}

	\maketitle
	
	\blfootnote{
		\begin{minipage}[l]{0.3\textwidth} \includegraphics[trim=10cm 6cm 10cm 5cm,clip,scale=0.15]{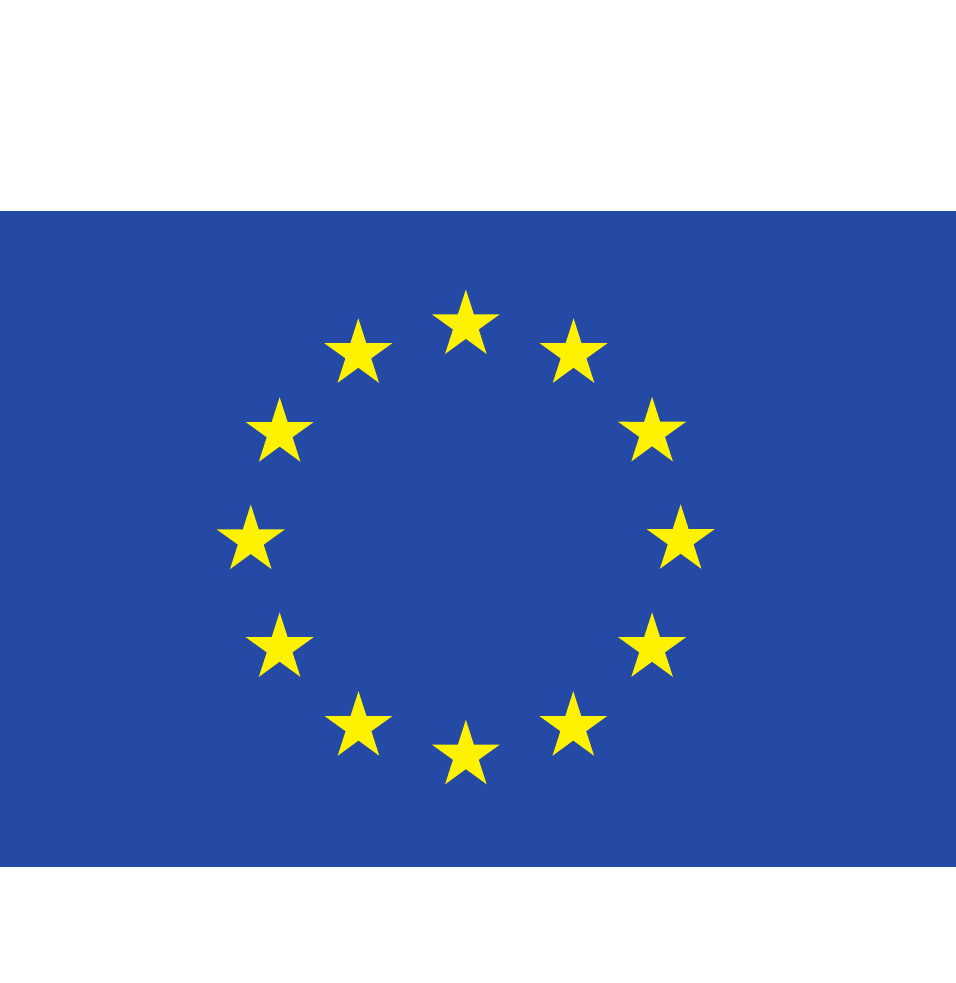} \end{minipage}  \hspace{-2cm} \begin{minipage}[l][1cm]{0.79\textwidth}
			The research of C. Dalf\'o has received funding from the European Union's Horizon 2020 research and innovation program under the Marie Sk\l{}odowska-Curie grant agreement No 734922.
		\end{minipage}
	}
	
	\begin{abstract}
		%\SZ{There is a strange red box on this page (link to footnote?) that should be removed}
		%\CR{No problem with the small red box, it's from TeX.}
		For $k\ge 1$, the $k$-independence number $\alpha_k$ of a graph is the maximum number of vertices that are mutually at distance greater than $k$. The well-known inertia and ratio bounds for the (1-)independence number $\alpha(=\alpha_1)$ of a graph, due to Cvetkovi\'c and Hoffman, respectively, were generalized recently for every value of $k$. 
		We show that, for graphs with enough regularity, the polynomials involved in such generalizations are closely related and give exact values for $\alpha_k$, showing a new relationship between the inertia and ratio type bounds.
		Additionally, we investigate the existence and properties of the extremal case of sets of vertices that are mutually at maximum distance for walk-regular graphs. Finally, we obtain new sharp inertia and ratio type bounds for partially walk-regular graphs by using the predistance polynomials.
		%Two vertices of a graph $G$, with $d+1$ distinct eigenvalues and spectrally maximum diameter $D=d$, are called maximally independent when they are at diatance $d$. Here we study the existence of $D$ of $d$-cliques, that is, sets of vertices which are mutually maximally independent, when $G$ is a walk.regular graph.
	\end{abstract}
	
	\vskip 1cm
	
	
	\section{Introduction}
	Given a graph $G$ with diameter $D$ and an integer $k\in [1,D-1]=\{1,\ldots, D-1\}$, the $k$-\emph{independence number} $\alpha_k=\alpha_k(G)$ of $G$ is the maximum number of vertices that are at distance greater than $k$ from each other.
	This is a natural generalization of
	the well-known independence  number $\alpha=\alpha_1$. In fact, the $k$-independence number of $G$ corresponds to the (standard)  independence number of the \emph{power graph} $G^k$, which has the same vertex set as $G$ and where two vertices $u$ and $v$ are adjacent in $G^k$ when they are at distance at most $k$ in $G$.
	However, even the simplest algebraic or combinatorial parameters of $G^k$ cannot be deduced easily from the similar parameters of $G$.  For instance, neither the spectrum (Das and Guo~\cite{Das2013LaplacianGraph}, and Abiad, Coutinho,  Fiol, Nogueira, Zeijlemaker \cite[Section 2]{acfnz21}), nor the average degree (Devos, McDonald, Scheide~\cite{Devos2013AveragePowers}), nor the rainbow connection number (Basavaraju, Chandran, Rajendraprasad, and Ramaswamy~\cite{Basavaraju2014RainbowProducts}) can be, in general, derived directly from the original graph. This provides the initial motivation of this work.

	The study of the $k$-independence number of a graph has received a considerable amount of attention. Kong and Zhao \cite{kz1993} showed that, for every $k\geq 2$, determining $\alpha_k(G)$ is an NP-complete problem for general graphs.
	They also proved that this problem remains NP-complete for regular bipartite graphs when $k\geq 2$ \cite{kz2000}. Duckworth and Zito extended a simple heuristic-based algorithm to approximate the independence number of connected random $d$-regular graphs to the $2$-independence number~\cite{DZ2003}. Hota, Pal, and Pal \cite{HPP2001} provided an efficient algorithm for finding a maximum weight $k$-independent set on trapezoid graphs. The $k$-independence number has also been studied in other contexts. For instance, Atkinson and Frieze \cite{AF2003} studied $\alpha_k$ in relation to random graphs. The mentioned complexity results on $\alpha_k$ provide motivation for finding tight bounds. Firby and Haviland \cite{fh97} showed lower and upper bound for $\alpha_k(G)$ in a connected graph on $n$ vertices as a function of $n$ and $k$. Beis, Duckworth, and Zito \cite{bdz2005} provided upper bounds for $\alpha_k(G)$ in random $r$-regular graphs for each fixed integers $k \geq 2$ and $r\geq 3$.
	O, Shi, and Taoqiu \cite{OShiTaoqiu2021} showed tight upper bounds for the $k$-independence number in an $n$-vertex
	$r$-regular graph for each positive integer $k\geq 2$ and $r\geq 3$ and with given minimum and maximum degrees.
	The case of $k=2$ has also received some attention: Jou, Lin, and Lin \cite{JLL2020} presented a tight upper bound for
	the $2$-independence number of a tree.
	Recently, Li and Wu \cite{lw2021} gave bounds for the $k$-independence number of a graph $G$ in terms of its order and vertex-connectivity. The $k$-independence number is also directly related to the study of distance-$j$ ovoids in incidence geometry. This study started with generalized polygons by Thas \cite{23}, who investigated
	the existence of distance-$2$ ovoids in generalized quadrangles and distance-$3$ ovoids in generalized
	hexagons (which are simply known as ovoids). The existence of distance-$j$ ovoids is related
	to the existence of particular perfect codes, see Cameron, Thas, and Payne \cite{4}, the separability of particular groups in Cameron and Kazanidis \cite{3}, and various other topics.
	
	A useful idea to find upper bounds for $\alpha_k(G)$ is to consider the (adjacency) spectrum of the graph $G$. The first eigenvalue bounds for the $k$-independence number, an inertia and a ratio type bound, were shown by
	Abiad, Cioab\u{a}, and Tait \cite{act2016}. Such inertia and ratio type bounds were improved further by Abiad, Coutinho, and Fiol \cite{acf19} by using more general polynomials that achieve equality for the corresponding bounds. The polynomials for the raio type bound were optimized by Fiol \cite{fiol20} in the case of partially walk-regular graphs. Analogously,  Abiad, Coutinho, Fiol, Nogueira, and Zeijlemaker \cite{acfnz21} formulated mixed integer linear programs (MILP) that optimize the choice of polynomials for the inertia type bound. Abiad, Elphick, and Wocjan \cite{ewa2020} proved that the inertia-type bound for the $k$-independence number also applies to the quantum $k$-independence number. The fact that this quantum parameter is not known to be computable justifies the use of optimization methods to find the exact value of the bounds. A common feature of all the aforementioned spectral bounds for $\alpha_k$ is that they involve various families of polynomials. Examples of this are the well-known predistance polynomials (which have been successfully applied for the characterization of distance-regular graphs, for instance, in the well-known Spectral Excess Theorem by Fiol and Garriga \cite{fg97}), or the sign and minor polynomials, see Fiol \cite{fiol20}.

	Despite the fact that both the inertia and ratio bounds for $\alpha$ provide simple spectral proofs for the celebrated Erdős–Ko–Rado Theorem, not much is known about the relationship between both bounds besides the fact that both use some type of interlacing. The results in \cite{ew,al2015} can all be viewed as attempts to provide unifications of such spectral bounds. In this work, we show that for graphs with enough regularity, the sign and the minor polynomials are sometimes closely related and provide exact values for $\alpha_k$. This provides a relationship between the known inertia and ratio type of bounds for $\alpha_k$. To do so, we introduce the notion of $k$-Cvetkovi\'c-Hoffman graphs to study when the two classes of polynomials that upper bound $\alpha_k$ are linearly related and when they both provide tight bounds.
	
	This paper is structured as follows. In Section \ref{sec:CHgraphs}, we focus on investigating $k$-Cvetkovi\'c-Hoffman graphs for the extreme cases: $k=1$ for regular graphs, and $k=d-1$ for walk-regular graphs (where $d$ is the number of distinct eigenvalues minus one). In the latter case, we consider the maximum cardinality of a set of vertices that are mutually at maximum distance, i.e. $d=D$ where $D$ is the diameter of $G$. In particular, in Section \ref{sec:maximallyindependentsets}, we study the existence and properties of $d$-cliques (also called $d$-spreads in the literature), that is, sets of vertices that are mutually at maximum distance $d$, when $G$ is a walk-regular graph. Our results provide a way to unify the inertia and ratio type bounds for the $k$-independence number. In particular, for $k$-partially walk-regular graphs, we show that both bounds for $\alpha_k$ can be seen as a linear combination of eigenvalue multiplicities. Note that, so far, a relationship between both bounds has only been found for $k=1$ for  graphs with enough regularity (see Haemers and Higman \cite{hh89}, who established it for graphs with exactly three distinct eigenvalues, thus showing a characterization of tight $1$-Cvetkovi\'c-Hoffman graphs that are strongly regular). In Sections \ref{sec:kpartiallywr} and \ref{sec:newboundspredistancepolys}, we prove that such a relationship can also be shown for $k$-partially walk-regular graphs. Additionally, 
	in Section \ref{sec:newboundspredistancepolys}, we use the predistance polynomials to provide new tight inertia and ratio bounds for $k$-partially walk-regular graphs.

	\section{Preliminaries}
	Let $G=(V,E)$ be a graph with $n=|V|$ vertices, $m=|E|$ edges, and adjacency matrix $A$ with spectrum
	$$
	\spec G =\spec A=\{\theta_0^{m_0},\theta_1^{m_1},\ldots,\theta_d^{m_d}\},
	$$
	where the different eigenvalues are, in decreasing order, $\theta_0>\theta_1>\cdots>\theta_d$ and the superscripts stand for their multiplicities. Since $G$ is assumed to be connected, we have that $m_0=1$.
	When the eigenvalues are presented with possible repetitions, we shall indicate them by
	$$
	\ev G:  \lambda_1 \geq \lambda_2 \geq \cdots\geq \lambda_n.
	$$
	
	%\subsection*{Partially walk-regular graphs}\label{sec:partiallywalkregular}
	%Let $G$ be a graph with diameter $D$.
	%The study of this paper concentrates on the case when $G$ has some regularity (see next). Thus, for this kind of graphs, we are interested in the maximum number of vertices mutually at maximum distance $D$.
	Let us now recall some known concepts.
	A graph $G$ is called {\em walk-regular} if the number of closed walks of any length from a vertex to itself does not depend on the choice of the vertex, a concept introduced by Godsil and McKay in \cite{gm80}.
	As a generalization used in Abiad, Coutinho, and Fiol \cite{acf19} and Abiad, Coutinho, Fiol, Nogueria, and Zeilemaker \cite{acfnz21}, a graph $G$ is called {\em $k$-partially walk-regular} for some integer $k\ge 0$, if the number of closed walks of a given length $l\le k$, rooted at a vertex $v$, only depends on $l$. In other words, if $G$ is $k$-partially walk-regular, for any polynomial $p\in \Real_k[x]$, the diagonal of $p(A)$ is constant with entries
	$$
	(p(A))_{uu}=\frac{1}{n}\tr p(A)=\frac{1}{n}\sum_{i=1}^n  p(\lambda_i)\quad \mbox{for all $u\in V$}.
	$$
	Thus, every (simple) graph is $k$-partially walk-regular for $k=0,1$ and every regular graph is $2$-partially walk-regular.
	A graph $G$ is called \emph{distance-regular} if for any two vertices $u$ and $v$ at distance $l$, the number of vertices at distance $i$ from $u$ and at distance $j$ from $v$, denoted by $p_{ij}^l$, depends only on $l$, $i$, and $j$.
	A graph $G$ is called $k$-\emph{partially distance-regular} if it is distance-regular up to distance $k$.
	Note that every $k$-partially distance-regular is $2k$-partially walk-regular. Moreover, $G$ is $k$-partially walk-regular for any $k$ if and only if $G$ is walk-regular. For example, it is well known that every distance-regular graph is walk-regular (but the converse does not hold). A $d$-regular graph $G$ on $n$ vertices is \emph{strongly regular} if every pair of adjacent (respectively, non-adjacent) vertices has $a$ (respectively, $c$) common adjacent vertices. Then, the parameters of $G$ are indicated by $(n,d,a,c)$. So, if connected, $G$ is distance-regular with diameter two.

	The first well-known spectral bound ({\em inertia bound}) for the independence number $\alpha=\alpha_1$ of $G$ is due to Cvetkovi\'c \cite{c71}:
	
	\begin{equation}
		\label{bound:cvetkovic}
		\alpha\le \min \{|\{i : \lambda_i\ge 0\}| , |\{i : \lambda_i\le 0\}|\}.
	\end{equation}
	When $G$ is regular, another well-known bound ({\em ratio bound}) for the independence number of $G$ is due to Hoffman (unpublished):
	\begin{equation}
		\label{bound:hoffman}
		\alpha \leq \frac{n}{1-\frac{\lambda_1}{\lambda_n}}.
	\end{equation}

	For the sake of comparison, in Table \ref{table:examples-alpha}, we show the values of the independence number, the (floor of the)  Hoffman bound \eqref{bound:hoffman}, and the Cvetkovi\'c  bound  \eqref{bound:cvetkovic} of some known regular graphs.
	Note that, in all these examples, the ratio bound is either equal to or smaller than the inertia bound,
	and this seems to be the case in general.
	However, there are graphs where the contrary happens, such as the Clebsh graph and the Higman-Sims graph. These graphs can be found in Table \ref{table:examples2-alpha}, where all  the (non-trivial) triangle-free strongly regular graphs are considered.
	The bounds of the Clebsh and Higman-Sims graphs are highlighted in bold. Note that, in both cases, their independence numbers coincide with their degree (the maximum independent sets are the neighborhoods of each vertex).
	In fact, this appears to be a very special property in strongly regular graphs. We checked it for strongly regular graphs with at most 250 vertices, and it only holds for the above two graphs and the pentagon (with parameters $(n,d,a,c)=(5, 2, 0, 1)$).

	\begin{table}
		\begin{center}
			\begin{tabular}{|c|c|c|c|}
				\hline
				Graph & $\alpha$ & Inertia bound & (Floor of) ratio bound \\
				\hline\hline
				Complete $K_n$ & 1 & 1 & 1\\
				\hline
				Halved Cube $K_n$ & 1 & 1 & 1\\
				\hline
				Cycle $C_5$ & 2 & 2 & 2\\
				\hline
				Halved Cube 16-cell & 2 & 3 & 2\\
				\hline
				Frankl–R\"odl $FR_{1/2}^4$ & 2 & 5 & 2\\
				\hline
				Circulant $(10;1,2)$ & 3 & 4 & 3\\
				\hline
				Petersen & 4 & 4 & 4\\
				\hline
				Prism$(5)=C_5\square P_2$ & 4 & 4 & 4\\
				\hline
				Shrikhande & 4 & 7 & 4\\
				%		\hline
				%		Clebsh & 5 & 5 & \textcolor{red}{6}\\
				\hline
				Hoffman & 8 & 11 & 8\\
				\hline
				Dodecahedron & 8 & 11 & 8\\
				\hline
				%		Circulant $(7;1,2)$ & 2 & 3 & 2\\
				%		\hline
				%		Circulant $(10;1,3,5)$ & 5 & 9 & 5\\
				%		\hline
				%		Hypercube $Q_3$ & 4 & 4 & 4\\
				%		\hline
				%		Hypercube $Q_4$ & 8 & 11 & 8\\
				%		\hline
				Middle Cube $MQ_3$ & 10 & 10 & 10\\
				\hline
				Desargues & 10 & 10 & 10\\
				\hline
				Coxeter & 12 & 12 & 12\\
				\hline
				Odd $O_3$ & 15 & 15 & 15\\
				\hline
				Hoffman-Singleton & 15 & 21 & 15\\
				\hline
				Hypercube $Q_5$ & 16 & 16 & 16\\
				\hline
			\end{tabular}
		\end{center}
		\caption{Some examples of the independence number, the inertia bound, and the ratio bound in some graphs.}
		\label{table:examples-alpha}
	\end{table}

	In addition to the Clebsh graph and Higman-Sims graph, the following strongly regular graphs from Brouwer's database \cite{Brouwer-srg} have smaller inertial bound than ratio bound:
	\begin{itemize}
		\item
		The McLaughlin Family (see Brouwer and van Lint~\cite{brouwerlint}): (112, 30, 2, 10), (162, 56, 10, 24), (243, 110, 37, 60), (275, 112, 30, 56), and (276, 140, 58, 84).
		\item $O^{-}(6,q)$, that is, the graphs whose vertices are points on an elliptic nondegenerate quadric in $PG(5,q)$, with adjacency defined by orthogonality: (27, 10, 1, 5), (112, 30, 2, 10), (325, 68, 3, 17), and (756, 130, 4, 26).
		\item
		The affine polar graphs $VO^{-}(4,q)$: (16, 5, 0, 2), (81, 20, 1, 6), (256, 51, 2, 12), (625, 104, 3, 20), and (2401, 300, 5, 42).
		\item
		The graphs derived from Lemma 5.3 in Godsil \cite{Godsil1992} with $r=2$: (27, 10, 1, 5), (125, 52, 15, 26), (343, 150, 53, 75), and (729, 328, 127, 164).
		\item
		The strongly regular graphs with parameters
		(105, 32, 4, 12),
		(120, 42, 8, 18),
		(126, 50, 13, 24),
		(175, 72, 20, 36),
		(176, 70, 18, 34),
		(253, 112, 36, 60) and
		(729, 112, 1, 20).
	\end{itemize}
	
	\begin{table}
		\begin{center}
			\begin{tabular}{|c|c|c|c|c|}
				\hline
				Graph & $(n,d,a,c)$ & $\alpha$ & Inertia bound & (Floor of) ratio bound \\
				\hline\hline
				Cycle $C_5$ & $(5,2,0,1)$ & 2 & 2 & 2\\
				\hline
				Petersen & $(19,3,0,1)$ & 4 & 4 & 4\\
				\hline
				Clebsh & $(16,5,0,2)$ & 5 & {\bf 5} & {\bf 6}\\
				\hline
				Hoffman-Singleton & $(50,7,0,1)$ & 15 & 21 & 15\\
				\hline
				Gewirtz & $(56,10,0,2)$ & 16 & 20 & 16\\
				\hline
				Mesner $M_{22}$ & $(77,16,0,7)$ & 21 & 21 & 21\\
				\hline
				Higman-Sims & $(100,22,0,6)$ & 22 & {\bf 22} & {\bf 26}\\
				\hline
			\end{tabular}
		\end{center}
		\caption{Some examples of the independence number, the inertia bound, and the ratio bound in the seven known triangle-free strongly regular graphs.}
		\label{table:examples2-alpha}
	\end{table}
	
	%
	%The bounds are equal for the complements of the triangular graphs $T_n = L(K_n)$ for all $n \ge 5$, for the Taylor 2-graphs for $U_3(q)$ with $q\in \{3,5,7,9\}$ and the strongly regular graphs with parameters (5, 2, 0, 1), (77, 16, 0, 4).
	
	\subsection{Two generalizations}
	We recall two known generalizations of the classic inertial and ratio bounds, which apply to the $k$-independence number, see Abiad, Coutinho, and Fiol \cite{acf19}.
	
	\begin{theorem}{(\cite{acf19})}
		\label{thmACF2019}
		Let $G$ be a graph with $n$ vertices and eigenvalues
		$\lambda_1\ge\cdots \ge \lambda_n$. Let  $p\in \Real_k[x]$ with  parameters $\lambda(p) = \min_{i\in[2,n]}\{p(\lambda_i)\}$, $W(p) = \max_{u\in V}\{(p(A))_{uu}\}$, and $w(p) = \min_{u\in V}\{(p(A))_{uu}\}$.
		\begin{itemize}
			\item[$(i)$]
			{\bf An inertial-type bound.}
			\begin{equation}
				\label{eq:thm1}
				\alpha_k\le \min\{|i : p(\lambda_i) \ge w(p)| , |i : p(\lambda_i) \le W(p)|\}.
			\end{equation}
			\item[$(ii)$] {\bf A ratio-type bound.}
			Assume that $G$ is regular. Let  $p\in \Real_k[x]$ such that $p(\lambda_1) > \lambda(p)$. Then,
			\begin{equation}
				\label{eq:thm2}
				\alpha_k\le n\frac{W(p)-\lambda(p)}{p(\lambda_1)-\lambda(p)}.
			\end{equation}
		\end{itemize}
	\end{theorem}
	
	Notice that the bounds \eqref{eq:thm1} and \eqref{eq:thm2} are invariant under scaling and/or translating the polynomial $p$.  Thus, it is enough to find a `good' polynomial satisfying the following:
	\begin{itemize}
		\item[$(i)$]
		In bound \eqref{eq:thm1}, a constant can be added to $p$ making $w(p)= \min_{u\in V}\{(p(A))_{uu}\}=0$. Moreover, multiplying the resulting polynomial by a (positive or negative) appropriate constant, it is enough to find a polynomial $p\in \Real_k[x]$ with fixed value of $\lambda(p)$,  say $-1$, and that minimizes the number of eigenvalues $\lambda_i$ such that $p(\lambda_i) \geq w(p)$. Altogether, \eqref{eq:thm1} becomes:
		\begin{equation}
			\label{eq:thm1'}
			\alpha_k\le \min_{i\in[1,n]} |\{i : p(\lambda_i) \ge 0\}|.
		\end{equation}
		\item[$(ii)$]
		Similarly, by an appropriate linear transformation, we can use a polynomial $p$ satisfying $p(\lambda_1)=1$ and $\lambda(p)=0$. Then \eqref{eq:thm2} becomes:
		\begin{equation}
			\label{eq:thm1'}
			\alpha_k\le nW(p).
		\end{equation}
		%By adding a constant to $p$ the value $\lambda(p)$ to our convenience.
	\end{itemize}

	%as done in Corollary \ref{coro:main} of the next subsection.
	
	\subsection{Partially walk-regular graphs}\label{sec:kpartiallywr}
	If $G$ is a $k$-partially walk-regular graph, we have $W(p)=w(p)=p(A)_{uu}=\frac{1}{n}\tr p(A)$ for any vertex $u \in V(G)$.  % or, in term of multiplicities, $\min \sum_{j: p(\theta_j)\leq W(p)} m_j$.,
	For such graphs, we therefore have the following result.
	\begin{corollary}
		\label{coro:main}
		Let $G$ be a $k$-partially walk-regular graph with diameter $D$ and spectrum $\spec G=\{\theta_0^{m_0},\theta_1^{m_1},\ldots,\theta_d^{m_d}\}$. Let $h(x)$ be the Heaviside function, that is, $h(x)=1$ if $x\ge 0$, and $h(x)=0$ otherwise. Then \eqref{eq:thm1} and \eqref{eq:thm2} can be stated as follows in terms of the multiplicities.
		\begin{itemize}
			\item[$(i)$]
			Let $p=s\in \Real_k[x]$ be a polynomial satisfying  $\lambda(s) = \min_{i\in[1,d]}\{s(\theta_i)\}= -1$ and $\tr s(A)=0$. Then,
			\begin{equation}
				\label{inertial-bound}
				\alpha_k\le \sum_{i=0}^d m_i h(s(\theta_i)).
			\end{equation}
			\item[$(ii)$]  Let  $p=f\in \Real[k]$ be a polynomial satisfying
			$\lambda(f)=\min_{i\in[1,d]}\{g(\theta_i)\}= 0$ and $f(\theta_0)=1$. Then,
			\begin{equation}
				\label{ratio-bound}
				\alpha_k\le \sum_{i=0}^d m_i f(\theta_i).%=\tr f(A).
			\end{equation}
		\end{itemize}
	\end{corollary}
	
	Since both results concern the same parameter, one would expect that a `good' polynomial $s$ provides a `good' polynomial $f$, and vice versa.
	%This is because of the following remark:
	\begin{remark}
		\label{remark}
		The following relations between the polynomials $s$ and $f$ hold.
		\begin{itemize}
			\item[$(i)$]
			If $s\in \Real_k[x]$ satisfies $\min_{i\in[1,d]}\{s(\theta_i)\}=-1$ and  $\tr s(A)=0$, then $f(x)=\frac{1+s(x)}{1+s(\theta_0)}$ satisfies $\min_{i\in[1,d]}\{f(\theta_i)\}=0$ and $f(\theta_0)=1$. Then, from \eqref{ratio-bound}, we get
			\begin{equation}
				\label{ratio-bound-2}
				\alpha_k\le \frac{n}{1+s(\theta_0)}.
			\end{equation}
			\item[$(ii)$]
			If $f\in \Real[k]$ satisfies $\min_{i\in[1,d]}\{f(\theta_i)\}=0$ and $f(\theta_0)=1$, then $s(x)=\frac{n}{\tr f(A)}f(x)-1$ has $\tr s(A)=0$ and  $\min_{i\in[1,d]}\{s(\theta_i)\}=-1$. Then, from \eqref{inertial-bound}, we get
			\begin{equation}
				\label{inertial-bound-2}
				\textstyle
				\alpha_k\le \sum_{i=0}^d m_i h\left( f(\theta_i)-\frac{1}{n}\tr f(A)\right) .
			\end{equation}
		\end{itemize}
	\end{remark}
	
	\subsubsection{Optimizing the upper bounds}
	\label{sec:optimizingboundsalphak}
	To optimize the upper bounds in Equation \eqref{eq:thm1} and \eqref{eq:thm2}, the following polynomials were introduced by Abiad, Coutinho, Fiol, Nogueira, and Zeijlemaker \cite{acfnz21} and Fiol \cite{fiol20}, respectively.
	
	\begin{description}
		\item[Bound \eqref{inertial-bound}:]
		If $s\in \Real[x]$  is a polynomial satisfying $\tr s(A)=0$, the best result is obtained by the so-called {\em sign polynomial} $s_k$, obtained as follows. Let $\b=(b_0,\ldots,b_d)\in \{0,1\}^{d+1}$ and $\m=(m_0,\ldots,m_d)$. For a given $k<D\ (\le d)$,  the  sign polynomial  $s_k(x) = a_k x^k +\cdots +a_0$ is the one with coefficients being the solution of the following MILP (mixed-integer linear programming) problem %\eqref{MILP:inertiaWR}
		with variables $a_1,\ldots,a_k$ and $b_0,\ldots, b_d$:
		\begin{equation}
			\boxed{
				\begin{array}{rl}
					% &	\\
					{\tt minimize} & \sum_{i=0}^d m_i b_i\\
					{\tt subject\ to} & \sum_{i = 0}^d m_i s_k(\theta_i)= 0\\
					& \sum_{i = 0}^k a_i \theta_j^{\ i} - Mb_j + \epsilon \leq 0, \quad j = 0,...,d\quad (\ast) \\
					& \b \in \{0,1\}^{d+1}
			\end{array}}
			\label{MILP:inertiaWR}
		\end{equation}
		Here $M$ is set to be a large integer and $\epsilon > 0$ a small number.
		The idea of this formulation is that each $b_j = 1$ represents an index $j$ so that $s_k(\theta_j) \geq w(p) = 0$. In fact, condition $(\ast)$ gives that $s_k(\theta_j)\ge 0$ implies $b_j=1$. So, upon minimizing the number of such indices $j$, we are optimizing $s_k(x)$ and the corresponding bound $\alpha_k\le  \sum_{i=0}^d m_i b_i$. See Abiad, Coutinho, Fiol, Nogueira, and Zeijlemaker  \cite{acfnz21} for more details.

		\item[Bound \eqref{ratio-bound}:]
		
		If $f\in \Real[k]$ is a polynomial satisfying $\lambda(f)=0$ and $f(\theta_0)=1$,  the best result is obtained with the so-called {\em minor polynomial} $f_k$ that minimizes  $\sum_{i=0}^d m_i f_k(\theta_i)$. % This case was studied by Fiol in \cite{fiol20}.
		This polynomial $f_k$ is defined by $f_k(\theta_0)=x_0=1$ and $f_k(\theta_i)=x_i$, for $i=1,\ldots,d$, where the vector $(x_1,x_2,\ldots,x_d)$ is a solution of the the following linear programming (LP) problem:
		\begin{equation}
			\boxed{
				\begin{array}{rl}
					{\tt minimize} & \sum_{i=0}^d m_ix_i\\
					{\tt subject\ to} & f[\theta_0,\ldots,\theta_m]=0,\quad m=k+1,\ldots,d\\
					& x_i\ge 0,\ i=1,\ldots,d\\
				\end{array}
			}
			\label{LPminorpoly}
		\end{equation}
		Here, $f[\theta_0,\ldots,\theta_m]$ denote the $m$-th divided differences of Newton interpolation, recursively defined by  $f[\theta_i,\ldots,\theta_j]=\frac{f[\theta_{i+1},\ldots,\theta_j]-f[\theta_i,\ldots,\theta_{j-1}]}
		{\theta_j-\theta_{i}}$, where $j>i$, starting with $f[\theta_i]=p(\theta_i)=x_i$ for $0\le i\le d$.
		Note that by equating these values to zero, we guarantee that $f_k\in \Real_k[x]$.
		%Then, we get
		%\begin{equation}
		%\label{chi_k-minor-pol}
		%\alpha_k\le \sum_{i=0}^d m_i p_k(\theta_i)=\tr p_k(A).
		%\end{equation}
	\end{description}
	
	Some known examples, properties, and approximations of the minor polynomials ({\bf `MP'}) are the following.
	For more details, see Fiol \cite{fiol20}.
	\begin{itemize}
		\item[{\bf MP0.}]
		For every $k=0,1,\ldots,d$, every $k$-minor polynomial $f_k$ has degree $k$ with its zeros in the interval $[\theta_d,\theta_0)\subset \Real$. Moreover, $f_k$ can always be chosen to have its $k$ zeros in the mesh $\{\theta_1,\ldots,\theta_d\}$ (see \cite[Prop. 3.2]{fiol20}).
		In the remainder of this section, we always choose $f_k$ such that it satisfies this condition.
		\item[{\bf MP1.}]
		$f_0(x)=1$ and $f_1(x)=\frac{x-\theta_d}{\theta_0-\theta_d}$. The degree one polynomial gives the bound $\alpha_1\le n\frac{-\theta_d}{\theta_0-\theta_d}$,  which corresponds to the Hoffman bound \eqref{bound:hoffman}.
		\item[{\bf MP2.}]
		$f_2(x)=\frac{(x-\theta_i)(x-\theta_{i+1})}{(\theta_0-\theta_i)(\theta_0-\theta_{i+1})}$, where $\theta_i$ is the smallest  eigenvalue greater than $-1$. This gives the following bound (see Abiad, Fiol, and Coutinho \cite{acf19}):
		\begin{equation}
			\label{bound-alpha2}
			\alpha_2\le n\frac{\theta_0+\theta_i\theta_{i+1}}{(\theta_0-\theta_i)(\theta_0-\theta_{i+1})}.
		\end{equation}
		\item[{\bf MP3.}]
		$f_3(x)\approx f_1f_2(x)=\frac{(x-\theta_i)(x-\theta_{i+1})(x-\theta_{i+1})}{(\theta_0-\theta_i)(\theta_0-\theta_{i+1})(\theta_0-\theta_{i+1})}$, where, as before, $\theta_i$ is the smallest eigenvalue such that $\theta_i>-1$ (in fact, the only possible zeros of $f_3$ are  $\theta_d$ and the consecutive pair $\theta_i,\theta_{i+1}$
		for some $i\in [1,d-1]$). Then the polynomial $f_1f_2$ gives the following bound (see Fiol \cite{fiol20}):
		\begin{equation}
			\label{bound-alpha3}
			\alpha_3\le n\frac{\Delta-\theta_i\theta_{i+1}\theta_d-\theta_0(\theta_i+\theta_{i+1}+\theta_d)}{(\theta_0-\theta_i)(\theta_0-\theta_{i+1})(\theta_0-\theta_d)},
		\end{equation}
		where $\Delta=w(x^3)=W(x^3)=(A^3)_{uu}$ for any $u\in V$ (recall that here we are assuming that $G$ is $3$-partially walk-regular).
		\item[{\bf MP4.}]
		$f_3(x)=\frac{1}{g(\theta_0)}(g(x)-\lambda(g))$, where $g(x)=x^3+bx^2+cx$ is the polynomial
		with coefficients $b=-(\theta_i+\theta_{i+1}+\theta_d)$ and $c=\theta_i\theta_d+\theta_{i+1}\theta_d+\theta_i\theta_{i+1}$, and $\theta_i$ is the smallest eigenvalue such that
		\begin{equation}
			\label{cond-theta-i}
			\theta_i\ge -\frac{\theta_0^2+\theta_0\theta_d-W(x^3)}{\theta_0(1+\theta_d)},
		\end{equation}
		where $W(x^3)=\max_{u\in V}\{(\A^3)_{uu}\}$
		(see Kavi, Newman, and Sajna \cite{kns21}). In fact, this exact formula for $f_3$ gives the same expression as Equation \eqref{bound-alpha3} for the bound of $\alpha_3$. This means that both polynomials, $f_1f_2$ and $f_3$, have zeros $\theta_i$, $\theta_{i+1}$, and $\theta_d$ where, in each case, $\theta_i$ is the smallest eigenvalue satisfying the condition in {\bf MP3} or {\bf MP4}, respectively. To compare both conditions, note that
		%\SZ{Both conditions have the sign in this direction `>'. How are they comparable?}

		%\AAM{What we try to say here is that when $\theta_i>-1$ (MP3), then $\theta_i$ also holds (15) (MP4). This, condition MP3 is weaker than MP4 (in principle), but we suggest that in most cases the resulting $\theta_i$ is the same and this, the approximation in MP3 can be tight.}
		
		$$
		-\frac{\theta_0^2+\theta_0\theta_d-W(x^3)}{\theta_0(1+\theta_d)}\le -1\quad \Leftrightarrow\quad  W(x^3)\le\theta_0(\theta_0-1).
		$$
		Since the condition on $W(x^3)$ holds for any (regular) graph, we infer that the condition in {\bf MP3} for the first zero of $f_1f_2$, that is $\theta_i>-1$,  implies condition \eqref{cond-theta-i} in {\bf MP4} for the first zero of $f_3$. In other words, the condition in MP3 is weaker than the one in MP4 (in principle), but in most cases the resulting $\theta_i$ could be the same, and so the approximation in MP3 could be exact. In this case, $f_3=f_1f_2$.
		\item[{\bf MP5.}]
		$f_{d-1}$ takes only one non-zero value at the mesh $\{\theta_1,\ldots,\theta_d\}$,
		which, because of the condition $\lambda(f_{d-1})=0$, must be located at some eigenvalue $\theta_i$ with odd index $i$. In fact, experimental evidence seems to suggest that such eigenvalue is either $\theta_1$ or one of $\theta_{d-1},\theta_d$ (depending on the parity of $d$). This gives the bound
		\begin{equation}
			\label{bound-alpha(d-1)}
			\alpha_{d-1}\le \min_{i \ {\rm odd}}\left\{1+m_i\frac{\pi_i}{\pi_0}\right\},
		\end{equation}
		where $\pi_i=\prod_{j\neq i}|\theta_i-\theta_j|$ (see  Fiol \cite{fiol20}).
		\item[{\bf MP6.}]
		$f_d(x)=\frac{1}{n}H(x)$, where $H$ is the Hoffman polynomial \cite{hof63} characterizing regularity by $H(A)=J$, the all-1 matrix.
	\end{itemize}
	
	\section{$k$-Cvetkovi\'c-Hoffman graphs}\label{sec:CHgraphs}
	From the above definitions, it seems interesting to know when, for a given value of $k$, the $k$-sign polynomial and the $k$-minor polynomial are linearly related according to Remark \ref{remark}. Moreover, if both polynomials give the same (inertia and ratio) bounds on the $k$-independence number, we can consider such polynomials to be `essentially' the same. A graph satisfying this property will be called a {\em $k$-Cvetkovi\'c-Hoffman graph} ({\em $k$-CH graph} for short).
	If the equal bounds are tight as well, we call it a {\em tight $k$-CH graph}. These definitions are motivated by the results of the following subsections concerning the cases $k=1$ and $k=d-1$.
	
	%%%%%%%%%%%%%%%%%%%%%%%%%%%%%%%%%%%%%%%%%%%%%%%%%%
	\subsection{The case $k=1$}
	We begin with two simple cases.
	
	\begin{lemma}
		\begin{itemize}
			\item[$(i)$]
			For any graph, the $1$-sign polynomial and the $1$-minor polynomial are linearly related.
			\item[$(ii)$]
			Every bipartite regular graph with even number $d+1$ of different eigenvalues is a tight 1-CH graph.
		\end{itemize}

		%\AAM{Rewrite $(i)$ as follows: (i) For any graph, the $1$-sign polynomial and the $1$-minor polynomial are linearly related.}
		
		\begin{proof}
			Equation \eqref{bound:cvetkovic} implies that the $1$-sign polynomial of $G$ is $s_1=\pm x$, whereas the $1$-minor polynomial is a linear function, as shown in {\bf MP1}.
			%\SZ{How does that show that they are equal. (also, `whereas' suggests a contradiction)}
			This proves the first statement.
			Moreover, under the hypothesis of $(ii)$, the spectrum of $G$ is symmetric around $0$ and it has no zeros. Hence, if $G$ it has $n$ vertices, both polynomials $s_1$ and $f_1$ give the exact independence number $\alpha=\frac{n}{2}$.
		\end{proof}
	\end{lemma}
	
	A more interesting result is the following characterization of tight 1-CH strongly regular graphs, due to Haemers and Higman in \cite[Theorem 2.6]{hh89}.
	
	\begin{theorem}[\cite{hh89}]\label{thm:higmanhaemers}
		Let $G$ be a strongly regular graph with maximum independent set $U\subset V$. Then both the inertia and ratio bounds are tight if and only if the graph induced by $\overline{U}=V\setminus U$ is strongly regular.
	\end{theorem}

	\subsection{The case of $k=d-1$}
	The other extreme case is when $k=d-1$. Then, we deal with maximally independent sets of vertices.
	%The aim of this subsection is to present the following results.
	For the case of triangle-free strongly regular graphs, we have the following result.
	\begin{lemma}
		If it exists, the triangle-free strongly regular graph $G(n)$ with feasible  parameters
		$$
		(n^4+5n^3+6n^2-n-1, n^2(n+2),0,n^2)\qquad \mbox{for  $n=1,2,\ldots$}
		$$
		is a tight $1$-CH graph.
	\end{lemma}
	\begin{proof}
		Since the spectrum of a strongly regular graph is closely related to its parameters (see, for instance, Godsil \cite{g93}), the proof is a simple computation.
	\end{proof}
	In fact, the only known graphs of this family are $G(1)=P$, the Petersen graph with parameters $(5,2,0,1)$, and $G(2)=M_{22}$, the Mesner graph with parameters $(77,16,0,4)$.
	Brouwer \cite{Brouwerprivate} told us that this family is mentioned in Brouwer and van Maldeghem \cite[Section 8.5.8]{bvm21} as subsconstituents of the `Krein graphs without triangles', which are only known for $r=1$, the Clebsch graph, and for $r=2$,  the Higman-Sims graph. For $r=3$, no example exists and nothing is known for $r>3$. According to Brouwer, it could be possible that $G(3)$ can be embedded in a $(324,57,0,12)$ strongly regular graph as second subconstituent, and then it would follow that such a graph does not exist.
	%A. L. Gavrilyuk & A. A. Makhnev,
	%On Krein Graphs without Triangles,
	%Dokl. Akad. Nauk 403 (2005) 727-730 (Russian) / Doklady Mathematics 72 (2005) 591-594 (English).
	%
	%Petteri Kaski & Patric R. J. Östergård,
	%There are exactly five biplanes with k = 11,
	%J. Combinatorial Designs 16 (2007) 117-127.
	
	%Often people are happy when they find a family of feasible parameters
	%where one or two of the parameters yield a famous sporadic graph.
	%Could it be that the other cases are just as interesting?
	%But that never happens. These sporadic groups are sporadic because
	%they are not part of a family.
	
	%I can imagine that it might be possible to show that your G(3) can be
	%embedded in a (324,57,0,12) graph as 2nd subconstituent, and then it
	%would follow that there is no such graph.
	
	Another infinite family of tight 1-CH is given by the Kneser graphs.
	Given integers $n$ and $k\le n$, the {\em Kneser graph} $K(n,k)$ has as vertices the ${n\choose k}$ $k$-subsets of the set $[1,n]$, and two vertices are adjacent if and only if their corresponding subsets are disjoint.
	For instance, $K(n,2)$ is
	the complement of the triangular graph $T_n = L(K_n)$ for all $n \ge 3$, where $L(G)$ represents the line graph of $G$.
	The Kneser graph $K(n,k)$ has order ${n\choose k}$, eigenvalues $\mu_j=(-1)^j{n-k-j\choose k-j}$ for $j=0,\ldots,k$, and multiplicities $m_0=1$ and  $m_j={n\choose j}-{n\choose j-1}$  for $j>0$.
	(Notice that, with this notation, the eigenvalues $\mu_0,\mu_1,\ldots,\mu_{k}$ are not in decreasing order. In particular, the above least eigenvalue $\theta_d$ is now $\mu_1$.)
	If $n=2k$, the Kneser graph $K(n,2k)$ is disconnected (constituted by different copies of $K_2$), and we omit this trivial case in the following result.

	\begin{corollary}\label{coro:EKR}
		The Kneser graph $K(n,k)$, with $n>2k$, is a tight 1-CH graph.
	\end{corollary}
	\begin{proof}
		From the above values of the eigenvalues and multiplicities, we find that both the inertia and ratio bounds (with the respective sign and minor polynomials) give ${n-1\choose k-1}$. Under the hypothesis and by the Erd\H{o}s-Ko-Rado Theorem \cite{ekr61} (see Katona \cite{k72} for a simple proof), this is known to be the exact value of the independence number $\alpha$ of $K(n,k)$. In fact, note that a maximum set of independent vertices can be formed by considering all $k$-subset containing some fixed digit in $[1,n]$.
	\end{proof}

	Other examples of tight 1-CH graphs are the Taylor 2-graphs for $U_3(q)$ with $q\in \{3,5,7,9\}$.

	In the next section, we prove, among others facts, the following results for graphs with larger diameter:
	
	\begin{itemize}
		\item[{\bf CH1.}]
		The Odd graph $O_{\ell}$ with even degree $\ell$ is a  $(d-1)$-CH graph. Moreover, the Odd graph $O_{\ell}$ is a tight $(d-1)$-CH graph for every even $\ell\in \{2,3,4,6,7,8,10,12,14,16\}$ (see Table \ref{table:odd-graphs(d-1)}).
		\item[{\bf CH2.}]
		The antipodal distance-regular graphs with odd diameter are tight $(d-1)$-CH graphs.
	\end{itemize}

	%%%%%%%%%%%%%%%%%%%%%%%%%%%%%%%%%%%%%%%%%%%%%%%%%%%%%%
	\section{Maximally independent sets}\label{sec:maximallyindependentsets}
	In this section, we focus on the extreme  case when $G$ is $(d-1)$-partially walk-regular (that is, walk-regular), and we search for the maximum cardinality of a set $U$  of vertices that are mutually at distance $d$ (that is, $k=d-1$).
	In the literature, such a set is known as a {\em $d$-spread} or {\em $d$-clique}. Thus, the $(d-1)$-independence number $\alpha_{d-1}$ coincides with the so-called the $d$-clique number or $d$-spread number $\omega_d$.
	In this context, Dalf\'o, Fiol, and Garriga \cite{dfg10} proved the following result.
	\begin{theorem}[\cite{dfg10}]
		Let $G$ be a walk-regular graph on $n$ vertices with spectrum $\spec G=\{\theta_0^{m_0},\ldots,\theta_d^{m_d}\}$ and spectrally maximum diameter $D=d$. Let $U\subset V$ be a $d$-clique with $r$ vertices. Then the set of projected points $\EE_iU=\{\EE_i\vece_u:u\in U\}$, where $\EE_i$ is the $i$-th idempotent representing the projection on the $\theta_i$-eigenspace  $\E_i\cong\Real^{m_i}$, are the vertices of an $(r-1)$-simplex  (that is, an $(r-1)$-dimensional polytope which is the convex hull of its $r$ vertices) in  $\E_i$, with barycenter $\c_r=\frac{1}{r}\sum \EE_i\vece_u$ at distance from $S$ to the origin, radius $R$ (distance from $\c_r$ to the origin), and edge length $L$ satisfying
		\begin{align}
			S &=\sqrt{\frac{1}{rn}\left(m_i+(-1)^i(r-1)\frac{\pi_0}{\pi_i} \right)},\\
			R &=\sqrt{\frac{r-1}{rn}\left(m_i+(-1) ^i\frac{\pi_0}{\pi_i} \right)},\\
			L &=\sqrt{\frac{2}{n}\left(m_i+(-1)^i\frac{\pi_0}{\pi_i} \right)},
		\end{align}
		where $\pi_i=\prod_{j\neq i}|\theta_i-\theta_j|$ for $i=0,\ldots,d$.
	\end{theorem}
	
	Since $R,L\ge 0$ and the maximum number
	of points mutually at a given distance equals $m_i+1$ in an $m_i$-dimensional space, we have the following consequence.
	
	\begin{corollary}
		Let $G$ be a walk-regular graph as above.  Let $U\subset V$ be a $d$-clique with $r$ vertices. Then the eigenvalue multiplicities and the $(d-1)$-independence number satisfy the following results.
		\begin{itemize}
			\item[$(i)$]
			If $i$ is even, then $m_i\ge \frac{\pi_0}{\pi_i}$. Besides, if $m_i\neq \frac{\pi_0}{\pi_i}$ ($R\neq 0$), then $\alpha_{d-1}\le 1+m_i$.
			\item[$(ii)$]
			If $i$ is odd, then $m_i \ge (r-1)\frac{\pi_0}{\pi_i}$ and $\alpha_{d-1}\le 1+m_i\frac{\pi_i}{\pi_0}$.
		\end{itemize}
		Moreover, equality for the multiplicity in $(i)$ is attained if and only if the simplex with vertices $\EE_i U$
		collapses into a point $(L=R=0)$, while  equality for the multiplicity in $(ii)$  is attained if the corresponding
		simplex is centered at the origin $(S=0)$.
	\end{corollary}
	
	Using our polynomials $s(x)$ and $f(x)$, we can obtain an alternative proof of the above upper bounds.
	First, notice that if $D<d$, then $\alpha_{d-1}=1$, so we may assume that $G$ has spectrally maximum diameter $D=d$.
	
	\begin{theorem}
		\label{th:alpha(d-1)}
		Let $G$ be a walk-regular graph with spectrum
		$\spec G =\{\theta_0^{m_0},\theta_1^{m_1},\ldots,\theta_d^{m_d}\}$. Let $\pi_i=\prod_{j=0,j\neq i}|\theta_i-\theta_j|$, for $i=0,\ldots,d$.
		Then the following holds.
		\begin{itemize}
			\item[$(i)$]
			For every $i=1,\ldots,\lfloor d/2\rfloor$ such that $m_{2i}\neq \frac{\pi_0}{\pi_{2i}}$,
			\begin{equation}
				\label{inertia-even}
				\alpha_{d-1}\le m_{2i}\quad \mbox{\bf (inertia bound)}.
			\end{equation}
			\item[$(ii)$]
			For every $i=1,\ldots,\lceil d/2\rceil$,
			\begin{equation}
				\label{inertia-odd}
				\alpha_{d-1}\le 1+ m_{2i-1}\quad \mbox{\bf (inertia bound)},
			\end{equation}
			and
			\begin{equation}
				\label{ratio-odd}
				\alpha_{d-1}\le 1+ m_{2i-1}\frac{\pi_{2i-1}}{\pi_{0}}\quad \mbox{\bf (ratio bound)}.
			\end{equation}
		\end{itemize}
		%$$
		%\alpha_{d-1}\le \min_{1\le i\le \lceil d/2\rceil} \{m_{2i},1+m_{2i-1}\}
		%\left\{
		%\begin{array}{cc}
		%1+m_{2i+1}
		%\end{array}
		%\right.
		%$$
	\end{theorem}
	\begin{proof}
		$(i)$
		Let $\tilde{s}\in \Real_{d-1}[x]$ be the monic polynomial with zeros at $\theta_j$ for $j\neq 0,2i$ for some $i=1,\ldots,\lfloor d/2\rfloor$, that is, $\tilde{s}(x)=\prod_{j\neq 0,2i}(x-\theta_{j})$. Then, $\tilde{s}(\theta_0)>0$, $\tilde{s}(\theta_{2i})<0$ and
		\begin{align}
			\tr \tilde{s}(A)& =\sum_{j=0}^d m_j \tilde{s}(\theta_j)=\tilde{s}(\theta_0)+m_{2i}\tilde{s}(\theta_{2i})
			=m_0\frac{\pi_0}{\theta_0-\theta_{2i}}+m_{2i}\frac{\pi_{2i}}{\theta_{2i}-\theta_0}\nonumber\\
			&=(\theta_0-\theta_{2i})^{-1}[\pi_0-m_{2i}\pi_{2i}].\label{tr-s-tilde}
		\end{align}
		Now, we claim that %(or $m_{2i}\ge \frac{\pi_{0}}{\pi_{2i}}$). Indeed, if $m_{2i}= \frac{\pi_{0}}{\pi_{2i}}$,
		$\tr \tilde{s}(A)\le 0$.
		By contradiction, if  $\tr \tilde{s}(A)>0$, there is a constant $\epsilon>0$ such that the polynomial $s'(x)=\tilde{s}-\epsilon$ has $\tr s'(A)=0$, and takes only one positive value at $\theta_0$. Then, Corollary \ref{coro:main}$(i)$ would imply that $\alpha_{d-1}\le 1$, contradicting that the diameter is $d$.  (Notice that the condition $\lambda(s')=-1$ is not necessary to apply to the corollary, although we could multiply $s'$ by a constant, if desired).
		Moreover, from  \eqref{tr-s-tilde} and the hypothesis on $m_{2i}$,  we know that $\tr \tilde{s}(A)\neq 0$.
		Hence, $\tr \tilde{s}(A)<0$ and there exists a constant $\sigma>0$ such that the polynomial $s(x)=\tilde{s}(x)+\sigma$ satisfies $\tr s(A)=0$, and it takes the only negative value at $\theta_{2i}$. See Figure \ref{fig:casei} for an example with $d=4$. Thus, the result follows from Corollary \ref{coro:main}$(i)$ by using the polynomial $-s(x)$.
		
		To prove $(ii)$, consider first the polynomial $\tilde{s}(x)=\prod_{j\neq 0, 2i-1}(x-\theta_j)$, now  satisfying  $\tr \tilde{s}(A)>0$ since $\tilde{s}(\theta_0),\tilde{s}(\theta_{2i-1})>0$. Then, there exists a constant $\tau>0$ such that the polynomial $s(x)=\tilde{s}(x)-\tau$ has $\tr s(A)=0$ and possible positive values only at $\theta_0$ and $\theta_{2i-1}$, as illustrated in Figure \ref{fig:caseii}. Thus, Corollary \ref{coro:main}$(i)$ gives \eqref{inertia-odd}.
		To prove \eqref{inertia-even}, we consider the polynomial $f(x)=\frac{\tilde{s}(x)}{\tilde{s}(\theta_0)}$, which satisfies the conditions in Corollary \ref{coro:main}$(ii)$, that is, $\lambda(f)=0$ and $f(\theta_0)=1$. Then, the result comes from $\alpha_{d-1}\le \tr f(A)$.
	\end{proof}

	\begin{figure}[!htb]
		\centering
		\begin{subfigure}[b]{0.45\textwidth}
			\centering
			\begin{adjustbox}{width=0.95\textwidth}
				\begin{tikzpicture}
					\node[draw = none, fill=none] at (0,-0.5){$\theta_2$};
					\node[draw = none, fill=none] at (-1.5,-0.5){$\theta_3$};
					\node[draw = none, fill=none] at (1.5,-0.5){$\theta_1$};
					\node[draw = none, fill=none] at (2.5,-0.5){$\theta_0$};
					\node[draw = none, fill=none] at (-2.6,-0.5){$\theta_4$};
					
					\node[fill=none,, text width=2cm] at (1.8,-1.3){$s = \tilde{s}+\sigma$};
					\node[fill=none, text width=2cm] at (1.8,-1.8){$\tilde{s}$};
					\node[fill=none] at (0,1.5){};
					\draw[dashed] (0.3,-1.3) -- (0.6,-1.3);
					\draw (0.3,-1.8) -- (0.6,-1.8);
					
					\draw (0,0.1) -- (0,-0.1);
					\draw (1.5,0.1) -- (1.5,-0.1);
					\draw (2.3,0.1) -- (2.3,-0.1);
					\draw (-1.5,0.1) -- (-1.5,-0.1);
					\draw (-2.6,0.1) -- (-2.6,-0.1);
					
					\node[circle, draw, thick, fill=none, inner sep=1pt] at (2.3,1.3984){};
					\node[circle, draw, thick, fill=none, inner sep=1pt] at (2.3,1.3984+0.1893){};
					\node[circle, draw, thick, fill=none, inner sep=1pt] at (1.5,0.1893){};
					\node[circle, draw, thick, fill=none, inner sep=1pt] at (-1.5,0.1893){};
					\node[circle, draw, thick, fill=none, inner sep=1pt] at (0,0.1893){};
					\node[circle, draw, thick, fill, inner sep=1pt] at (-2.6,-2.3452){};
					\node[circle, draw, thick, fill, inner sep=1pt] at (-2.6,-2.3452+0.1893){};
					
					\draw (-2.7, 0) -- (2.4, 0) node[right] {};
					\draw[scale=1, domain=-2.7:2.4, smooth, variable=\x] plot ({\x}, {0.2*(\x-1.5)*(\x)*(\x+1.5)});
					\draw[scale=1, domain=-2.7:2.4, smooth, variable=\x, dashed] plot ({\x}, {0.2*(\x-1.5)*(\x)*(\x+1.5)+0.1893});
				\end{tikzpicture}
			\end{adjustbox}
			\caption{$d=4$, $i=2$}
			\label{fig:casei}
		\end{subfigure}
		\begin{subfigure}[b]{0.45\textwidth}
			\centering
			\begin{adjustbox}{width=0.95\textwidth}
				\begin{tikzpicture}
					\node[draw = none, fill=none] at (0,-2.2){};
					
					\node[draw = none, fill=none] at (0,-0.6){$\theta_2$};
					\node[draw = none, fill=none] at (-1,-0.6){$\theta_3$};
					\node[draw = none, fill=none] at (1,-0.6){$\theta_1$};
					\node[draw = none, fill=none] at (1.4,-0.6){$\theta_0$};
					\node[draw = none, fill=none] at (-2,-0.6){$\theta_4$};
					\node[draw = none, fill=none] at (-2.4,-0.6){$\theta_5$};
					
					\node[fill=none,draw=none] at (0,-2.6){}; %align pictures
					
					\node[fill=none, text width=2cm] at (1.5,-1.3){$\tilde{s}$};
					\node[fill=none, text width=2cm] at (1.5,-1.8){$s = \tilde{s}-\tau$};
					\node[fill=none] at (-0.5,1.5){};
					\draw (0,-1.3) -- (0.3,-1.3);
					\draw[dashed] (0,-1.8) -- (0.3,-1.8);
					
					\draw (0,0.1) -- (0,-0.1);
					\draw (1,0.1) -- (1,-0.1);
					\draw (-2,0.1) -- (-2,-0.1);
					\draw (-1,0.1) -- (-1,-0.1);
					\draw (-2.4,0.1) -- (-2.4,-0.1);
					\draw (1.4,0.1) -- (1.4,-0.1);
					
					\node[circle, draw, thick, fill=none, inner sep=1pt] at (-2.4,0.91392){};
					\node[circle, draw, thick, fill=none, inner sep=1pt] at (1.4,0.91392){};
					\node[circle, draw, thick, fill=none, inner sep=1pt] at (-2.4,0.91392-0.3046){};
					\node[circle, draw, thick, fill=none, inner sep=1pt] at (1.4,0.91392-0.3046){};
					\node[circle, draw, thick, fill, inner sep=1pt] at (-2,-0.3046){};
					\node[circle, draw, thick, fill, inner sep=1pt] at (1,-0.3046){};
					\node[circle, draw, thick, fill, inner sep=1pt] at (0,-0.3046){};
					\node[circle, draw, thick, fill, inner sep=1pt] at (-1,-0.3046){};
					
					\draw (-3.2, 0) -- (2.1, 0) node[right] {};
					\draw[scale=1, domain=-2.5:1.5, smooth, variable=\x,dashed] plot ({\x}, {0.2*(\x-1)*(\x)*(\x+2)*(\x+1)-0.3046});
					\draw[scale=1, domain=-2.5:1.5, smooth, variable=\x] plot ({\x}, {0.2*(\x-1)*(\x)*(\x+2)*(\x+1)});
				\end{tikzpicture}
			\end{adjustbox}
			\caption{$d=5$, $i=2$}
			\label{fig:caseii}
		\end{subfigure}
		\caption{An illustration of Case (i) and (ii) of Theorem \ref{th:alpha(d-1)}}
		\label{fig:casesTh43}
	\end{figure}
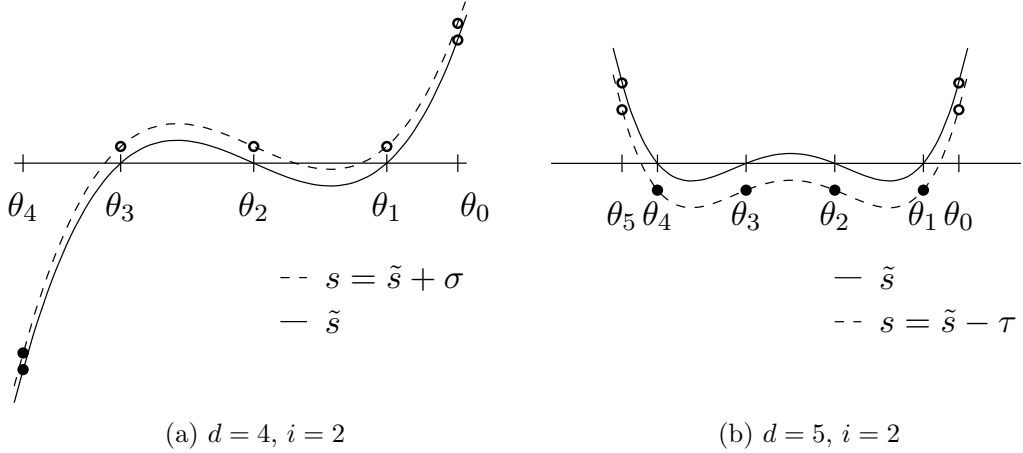
	
	Of course, in $(i)$ we could derive a ratio bound by applying Corollary \ref{coro:main}$(ii)$ with the polynomial $f(x)=\frac{\tilde{s}(x)+\tilde{s}(\theta_{2i})}{\tilde{s}(\theta_0)+\tilde{s}(\theta_{2i})}$,
	but the result is not good,
	%\SZ{Do you mean that it doesn't work? In that case I wouldn't formulate it as `we could derive', because we cannot.}
	since $\lambda(f)>0$ (instead of $\lambda(f)=0$, like it should be when we use the optimal minor polynomials.)
	%\SZ{This sentence between brackets is very confusing}). 
	In fact, the minor polynomial $f_{d-1}$ takes only non-zero values at $\theta_0$ and $\theta_i$ for some $i$ odd, as in case $(ii)$.
	
	Note that in $(ii)$, one of the bounds \eqref{inertia-odd} or \eqref{ratio-odd} is better than the other depending on whether $\pi_{2i-1}>\pi_0$ or $\pi_{2i-1}<\pi_0$.
	However, it appears that the second bound is always either equal or better than the first one. Although sometimes $\frac{\pi_{2i-1}}{\pi_0} > 1$, there always seem to be $i,j$ such that $m_{2i-1} \ge \frac{\pi_{2j-1}}{\pi_0}m_{2j-1}$. For instance, we checked all small graphs up to seven vertices, all Sage's named graphs and many well-known graph families, and this is always the case. Moreover, the bounds coincide when $\pi_{2i-1}=\pi_0$. As we will see later, this happens for the Odd graphs $O_{\ell}$ with even $\ell=d+1$ and $2i-1=d$.
	In fact, the exact value of $\frac{\pi_{2i-1}}{\pi_0}$ can be easily computed for some families of graphs. The following list gives some examples.
	\begin{itemize}
		\item
		For even cycles and complete bipartite graphs, we have $\frac{\pi_{2i-1}}{\pi_0} = \frac{1}{2}$ for any $i$.
		\item
		The 6-cages of valency $q+1$ ($q$ prime) have $\frac{\pi_{1}}{\pi_0} = \frac{\sqrt{q}}{q+1}$ and $\frac{\pi_{3}}{\pi_0}=1$. Since $\theta_1$ has large multiplicity, $1+\frac{\pi_{3}}{\pi_0}m_3 = 2$ is still the better upper bound.
		\item
		The hypercubes have $\frac{\pi_{j}}{\pi_0} = \binom{n}{j}^{-1}$, hence $\frac{\pi_{j}}{\pi_0}m_j = 1$ for all $j = 2i-1$.
	\end{itemize}
	
	Notice also that if the diameter of $G$ is smaller than $d$, then $\alpha_{d-1}=1$ and the above result is trivial. On the other hand, for the case of walk-regular graphs with spectrally maximum diameter $d$, the maximum number of vertices that are mutually at distance $d$ satisfies the following result.
	
	\begin{corollary}
		\label{coro:md}
		Let $G$ be a walk-regular graph with spectrum
		$\spec G =\{\theta_0^{m_0},\theta_1^{m_1},\ldots,\theta_d^{m_d}\}$ and diameter $D=d$.
		\begin{itemize}
			\item[$(i)$]
			If $d$ is even, then
			$
			\alpha_{d-1}\le m_{d}.
			$
			\item[$(ii)$]
			If $d$ is odd, then
			$
			\alpha_{d-1}\le 1+ m_d\cdot \min\{1,\frac{\pi_d}{\pi_0}\}.
			$
		\end{itemize}
	\end{corollary}
	
	Next, we will show that the bounds of Corollary \ref{coro:md} are tight for some Odd graphs. First, recall that, for every integer $\ell\geq 2$, the \emph{Odd graph} $O_{\ell}$ can be defined as the Kneser graph $K(2\ell-1,\ell-1)$. In other words, $O_{\ell}$ has vertices  corresponding to the $(\ell-1)$-subsets of a $(2\ell -1)$-set, and the adjacencies are defined by void intersection. Note that $O_3$ is the Petersen graph $P$. In general, $O_{\ell}$  is an $\ell$-regular graph of order $n=\binom{2\ell -1}{\ell -1}=\frac{1}{2}\binom{2\ell}{\ell}$, diameter $D=\ell-1$, and its eigenvalues and multiplicities are $\mu_i=(-1)^i(\ell-i)$ and $m(\mu_i)=m_i=\binom{2\ell-1}{i}-\binom{2\ell-1}{i-1}$ for $i=0,1,\ldots, \ell-1$, so that $d=D=\ell-1$. As for general Kneser graphs, these eigenvalues $\mu_0,\mu_1,\ldots,\mu_d$ are again not in decreasing order, and the least eigenvalue is $\mu_1$. The following result was shown by Fiol \cite{fiol20} and Abiad, Coutinho, Fiol, Nogueira, and Zeijlemaker \cite{acfnz21}.
	
	\begin{proposition}[\cite{fiol20,acfnz21}]
		\label{propo:alphaOl}
		For $i=0,\ldots,\ell-1$, let $\mu_i$ and $m_i$ be the eigenvalues and multiplicities of the Odd graph $O_{\ell}=O_{d+1}$.
		%,defined as above.
		Then,
		\begin{equation}
			\label{propo:known1}
			\alpha_{d-1}(O_{d+1})\le \left\{
			\begin{array}{ll}
				m_1 & \mbox{for even $d$} \\
				m_1+1 & \mbox{for odd $d$}
			\end{array}
			\right\}=
			\left\{
			\begin{array}{ll}
				2d & \mbox{for even $d$,} \\
				2d+1 & \mbox{for odd $d$.}
			\end{array}
			\right.
		\end{equation}
	\end{proposition}
	
	Next, we provide an alternative proof of Proposition \ref{propo:alphaOl} using our polynomials. Moreover, we also show that, for $\ell$ even, the Odd graphs $O_{\ell}$ are $(d-1)$-CH graphs (that is, the inertia and ratio type bounds for $\alpha_{d-1}(O_{\ell})$ coincide and are obtained by sign and minor polynomials that are linearly related), as stated in {\bf CH1} at the end of the previous section.
	
	\begin{proposition}
		%With the same notation as in Proposition \ref{propo:alphaOl}, the following statements hold.
		For $i=0,\ldots,\ell-1$, let $\mu_i$ and $m_i$ be the eigenvalues and multiplicities of the Odd graph $O_{\ell}=O_{d+1}$.
		\begin{itemize}
			\item[$(i)$]
			The $(d-1)$-independence number of $O_{d-1}$ satisfies
			\begin{equation}
				\label{propo:known2}
				\alpha_{d-1}(O_{d+1})\le 1+m_2\frac{\pi_2}{\pi_0}=
				\left\{
				\begin{array}{ll}
					2d+\frac{d-2}{d} & \mbox{for even $d$,} \\
					2d+1 & \mbox{for odd $d$.}
				\end{array}
				\right.
			\end{equation}
			\item[$(ii)$]
			If $\ell$ is even, $O_{\ell}$ is a $(d-1)$-CH graph.
			Moreover, in this case,
			\begin{equation}
				\label{simple-eq}
				\frac{m_1}{m_2}=\frac{\pi_2}{\pi_0}.
			\end{equation}
		\end{itemize}
	\end{proposition}
	
	\begin{proof}
		$(i)$ This is because the minor polynomial $f_{d-1}=f_{\ell-2}$ of $O_{d+1}=O_{\ell}$ can be taken as $f_{d-1}(x)=\frac{\prod_{i\neq 0,2}(x-\mu_i)}{\prod_{i\neq 0,2}(\mu_0-\mu_i)}$, which gives
		$$
		\tr f_{d-1}(A)= 1+m_2f_{d-1}(\mu_2)=1+m_2\frac{\pi_2}{\pi_0}=
		\left\{
		\begin{array}{ll}
			2d+\frac{d-2}{d} & \mbox{for even $d$,} \\
			2d+1 & \mbox{for odd $d$.}
		\end{array}
		\right.
		$$
		
		$(ii)$ For every Odd graph $O_{\ell}$ with even $\ell$ (that is, odd $d$), we can also take the minor polynomial $f_{d-1}$ whose only non-zeros are at $\theta_0=\ell$ and $\theta_d=-\ell+1$, giving again the bound $\alpha_{d-1}\le 2d+1$. In fact, this is the same bound obtained by the sign polynomial that, following Remark \ref{remark}$(ii)$, can be written as
		$s_{d-1}(x)=\frac{n}{\tr f_{d-1}(A)}f_{d-1}(x)-1$.
		Hence, the minor and sign polynomials are linearly related.
		Finally, \eqref{simple-eq} comes from equating the bounds in \eqref{propo:known1} and \eqref{propo:known2} when $\ell$ is even.
	\end{proof}
	
	The values $x_i=f_k(\theta_i)$ at the different eigenvalues $\theta_d<\theta_{d-1}<\cdots <\theta_0$  of the Odd graphs $O_5$ and $O_6$ are shown in Table \ref{table1}.
	
	%\begin{figure}[h!]
	%\begin{center}
	%\includegraphics[width=13cm]{polsHamming}
	%\vskip-7.5cm
	%\caption{The minor polynomials of the odd graphs $O_5$ and $O_6$.}
	%\label{fig1}
	%\end{center}
	%\end{figure}
	
	\begin{table}[h!]
		\begin{center}
			\begin{tabular}{|c|ccccc|c|}
				\hline
				$k$ &  $x_4$ & $x_3$ & $x_2$ & $x_1$ & $x_0$ & $\tr f_k(A)$\\
				\hline\hline
				1 & 0 & 2/9 & 5/9 & 7/9 & 1  & 56\\
				\hline
				2 & 5/14 & 0 & 0  & 5/14 & 1 & $13.5$ \\
				\hline
				3 & $0$  & 0 & 0 & 5/18 & 1 & $8.5$ \\
				\hline
			\end{tabular}
		\end{center}
		\begin{center}
			\begin{tabular}{|c|cccccc|c|}
				\hline
				$k$ &  $x_5$ & $x_4$ & $x_3$ & $x_2$ & $x_1$ & $x_0$ & $\tr f_k(A)$\\
				\hline\hline
				1 & 0 & 2/11 & 4/11 & 7/11 & 9/11  & 1  & 210\\
				\hline
				2 & $1$ & 5/14 & 0 & 0  & 5/14 & 1 & 66 \\
				\hline
				3 & $0$  & 5/77  & 0 & 0 & 45/154 & 1 & 21 \\
				\hline
				4 & $1$  & 0 & 0 & 0  & 0 &  1 & 11\\
				\hline
			\end{tabular}
		\end{center}
		\caption{Values $x_i=f_k(\theta_i)$ for $i=0,\ldots, d$ of the $k$-minor polynomials of the Odd graphs $O_5$ and $O_6$.}
		\label{table1}
	\end{table}

	If the bounds are tight,
	%(see what follows\SZ{vague}), 
	there are some interesting examples where the inertial and ratio bound coincide (and, as commented, the involved polynomials are linearly related according to Remark \ref{remark}).
	For instance, in the Odd graph $O_6$, the minor polynomial for $k=4$ is
	$$
	\textstyle
	f_4(x)=\frac{1}{504}(x^4-2x^3-13x^2+14x+24),
	$$
	whereas the corresponding sign polynomial is
	$$
	\textstyle
	s_4(x)=\frac{1}{12}(x^4-2x^3-13x^2+14x+12).
	$$
	In Table \ref{table2}, we show their values at the mesh $\theta_5,\theta_4,\ldots,\theta_0$, together with the value of the traces
	of the matrices when evaluated at $A$. Then, the ratio bound for $\alpha_{4}$ is $11$, which coincides with the inertia bound $m_0+m_5=11$.
	
	\begin{table}[h!]
		\begin{center}
			\begin{tabular}{|c|cccccc|c|}
				\hline
				pol  &  $\theta_5=-5$ & $\theta_4=-3$ & $\theta_3=-1$ & $\theta_2=2$ & $\theta_1=4$ & $\theta_0=6$ & $\tr \mbox{\rm pol}(A)$\\
				\hline\hline
				$f_4(x)$ & 1  & 0 & 0 & 0 & 0 & 1  & 11\\
				\hline
				$s_4(x)$ & $41$ & $-1$ & $-1$  & $-1$ & $-1$ & $41$ & $0$ \\
				\hline
			\end{tabular}
		\end{center}
		\caption{Values of the minor polynomial $f_4(x)$ and sign polynomial $s_4(x)$ of the Odd graph $O_6$.}
		\label{table2}
	\end{table}

	In Abiad, Coutinho, Fiol, Nogueira, and Zeijlemaker \cite{acfnz21}, it was shown that the bounds of Proposition \ref{propo:alphaOl} are tight for $\ell\in\{4,6,7,8,10,12,14\}$, and this also holds for $\ell=2$ ($K_3$) and $\ell=3$ (the Petersen graph).
	However, this is not the case when $\ell\in\{5,9,11\}$. The known exact values of $\alpha_{d-1}$ and the corresponding upper bounds are shown in  Table \ref{table:odd-graphs(d-1)}. The exact values for odd $\ell$ were found by computer search. However, when $\ell$ is even, the exact value of $\alpha_{d-1}$ can be proved theoretically through its relation with symmetric designs.

	%\AAM{We forgot to mention some known cases of Hadamard designs which correspond to tightness in our upper bound.
		%	The graphs we care about are the ones with vertices the $d$-sets in an $(2d+1)$-set, and they are non-adjacent if they pairwise meet in exactly a (d/2)-set (for d even) or a $(d-1)/2$-set (for d odd). So the question becomes: for which values does there exist Hadamard designs? Since, if the bound is reached for d odd, then the family of vertices corresponds to such a Hadamard design. And apparently, Hadamard designs exist for all $d+1=2^n$, so for these, the bound is sharp (we missed this case in our previous paper). But then the question we need to mention in this new paper is: what is known for other odd values for d?
		%	To prove that it is sharp (for odd d), it is equivalent to prove that there exists a Hadamard design with these parameters. But, as far as I know it is only known for $d+1=2^n$ that these designs exists, see:	
		%\href{https://en.wikipedia.org/wiki/Hadamard_matrix}
		%	and this is equivalent to the Hadamard conjecture. So is not only up to $O_{16}$, and hence, the paragraph after Propo 4.6 needs some rewriting.}

	Let $\ell\geq 4$ be even. Then the bound in Proposition \ref{propo:alphaOl} is tight, that is, \ $\alpha_{\ell-2}(O_\ell) = \ell -1$, if and only if the vertices of a maximum $(\ell-2)$-independent set constitute a $2$-$(2\ell -1, \ell -1, \frac{1}{2}\ell -1)$ symmetric design (see Hall \cite{hall1986} for its definition). 
	In terms of intersecting set systems, this is equivalent to finding the largest system of $(\ell-1)$-subsets in a $(2\ell-1)$-set such that the intersection of any two sets has size $\ell/2 - 1$.
	Such designs are known to exist for $\ell = 4, 6, \dots, 16$ (see Stinton \cite{stinson2007}), which correspond to the entries in Table  \ref{table:odd-graphs(d-1)}. Moreover, there exists a Hadamard matrix of order $4m$ if and only if there exists a symmetric 2-$(4m - 1, 2m - 1, m - 1)$-design, for $m>1$ (see Stinton \cite[Th. 4.5]{stinson2007} again). It is known that Hadamard matrices of order $4m$ exist whenever $4m=2^n$. Then, the bound in Proposition \ref{propo:alphaOl} is also tight for every $\ell$ a power of 2.
	It would be interesting to know other exact values in order to show a more general result for other values of $\ell$. 
	%In particular, we wonder whether there exists such a set system of size $2\ell-1$ for every even $\ell$. 
	
	%\AAM{Here we should also mention why one cannot use the known bounds for $\alpha$ on $O^{\ell}$}

	\begin{table}
		\begin{center}
			\begin{tabular}{|c|r|r|}
				\hline
				Graph & $\alpha_{d-1}$ & Bound\\
				\hline\hline
				$O_2$ ($K_3$) & $\alpha_0=3$ &  $m_0+m_1=3$\\
				\hline
				$O_3$ (Petersen) & $\alpha_1=4$ & $m_1=4$\\
				\hline
				$O_4$ & $\alpha_2=7$ &  $ m_0+m_1 =7$\\
				\hline
				$O_5$ & $\alpha_3=7$ &  $m_1=8$ \\
				\hline
				$O_6$ & $\alpha_4=11$ &  $m_0+m_1=11$\\
				\hline
				$O_7$ & $\alpha_5=12$ & $m_1= 12$\\
				\hline
				$O_8$ & $\alpha_6=15$ & $m_0+m_1=15$\\
				\hline
				$O_9$ & $\alpha_7=15$ & $m_1=16$\\
				\hline
				$O_{10}$ & $\alpha_8=19$ & $m_0+m_1=19$\\
				\hline
				$O_{11}$ & $\alpha_9=19$ & $m_1=20$\\
				\hline
				$O_{12}$ & $\alpha_{10}=23$ & $m_0+m_1=23$\\
				\hline
				$O_{14}$ & $\alpha_{12}=27$ & $m_0+m_1=27$\\
				\hline
				$O_{16}$ & $\alpha_{14}=31$ & $m_0+m_1=31$\\
				\hline
			\end{tabular}
		\end{center}
		\caption{The known exact values of $\alpha_{d-1}$ and the upper bounds for the Odd graphs $O_{d+1}$.}
		\label{table:odd-graphs(d-1)}
	\end{table}

	%\hline
	%\multirow{3}{*}{$\alpha_6(O_{8})$} & Bound from the MILP & $15=m_0+m_1$ \\
	% & Polynomial $\sigma^6+7\sigma^5-45\sigma^4-287\sigma^3+256\sigma^2+1372\sigma-144$  &  0.1032025452 \\
	%% & Bound from \cite{fiol20} & 15 \\
	% & Exact value $\alpha_6$ & 15 \\
	%\hline
	%\hline
	%$\alpha_8(O_{10})$ & Bound from the MILP & $19=m_0+m_1$ \\
	%  &  Exact value $\alpha_8$ & 19 \\
	%\hline
	%$\alpha_{10}(O_{12})$ & Bound from the MILP & $23=m_0+m_1$ \\
	% &  Exact value $\alpha_{10}$ & 23 \\
	%\hline
	%$\alpha_{12}(O_{14})$ & Bound from the MILP & $27=m_0+m_1$ \\
	% &  Exact value $\alpha_{12}$ & 27 \\
	%\hline
	%\end{tabular}
	% \caption{Infinite family of Odd graphs for which the output from MILP \eqref{MILP:inertiaWR} gives the best polynomials for upper bounding $\alpha_k$.} \label{tableOddgraph}
	% \end{table}

%\begin{conjecture}
%\label{conjecture}
%Excepting for $\ell=d+1=5$, the $(d-1)$-independence number of the odd graph $O_{\ell}$ equals $2d$ if $d$ is even and $2d+1$ if $d$ is odd.
%\end{conjecture}

Next, we extend a result by Dalf\'o, Fiol, and Garriga \cite{dfg10} by showing that every antipodal distance-regular graph with odd diameter is a tight $(d-1)$-CH graph  (see {\bf CH2}).

\begin{proposition}
	\label{prop:multiplicities}
	Let $G$ be a walk-regular graph with spectrum
	$
	\spec G =\{\theta_0^{m_0},\theta_1^{m_1},\ldots,\theta_d^{m_d}\}
	$, diameter $D=d$, and $(d-1)$-independence number $\alpha_{d-1}=r$. Then, for $i=1,\ldots, \lceil d/2\rceil$, the multiplicities satisfy the bounds
	\begin{equation}
		\label{bounds-mult}
		m_{2i}\ge \frac{\pi_0}{\pi_{2i}}\qquad\mbox{and}\qquad m_{2i-1}\ge (r-1)\frac{\pi_0}{\pi_{2i-1}}.
	\end{equation}
	Moreover, if the mean number of vertices at distance $d$ from every vertex equals $r-1$, equalities hold in \eqref{bounds-mult} for every  $i=1,\ldots, \lceil d/2\rceil$ if and only if $G$ is an $r$-antipodal distance-regular graph.
\end{proposition}
\begin{proof}
	Note that in the proof of Theorem \ref{th:alpha(d-1)}, we already showed that $m_{2i}\ge \frac{\pi_0}{\pi_{2i}}$.
	Finally, $m_{2i-1}\ge (r-1)\frac{\pi_0}{\pi_{2i-1}}$ is a consequence of \eqref{ratio-odd}. %(In fact, a different proof of \eqref{bounds-mult} was given in Dalf\'o, Fiol, and Garriga \cite{dfg10}.)
	
	For the case of equality, let us first show that $G$ is distance-regular. For this, we can use the spectral excess theorem by Fiol and Garriga \cite{fg97}, which states that a regular graph is distance-regular if and only if the spectral excess $p_d(\theta_0)$ (see Eq. \eqref{eq:spectral-excess})
	%below\SZ{where?}) 
	equals the average excess $\overline{k}_d$ (the mean of the number of vertices at distance $d$ from each vertex, in our case $r-1$).
	Note that $G$ has order
	\begin{equation}
		\label{order}
		n=\sum_{i=0}^d m_i=\sum_{\mbox{\scriptsize $i$ even}} \frac{\pi_0}{\pi_i}+(r-1)\sum_{\mbox{\scriptsize $i$ odd}}\frac{\pi_0}{\pi_i}=r\sum_{\mbox{\scriptsize $i$ odd}}\frac{\pi_0}{\pi_i}=\frac{r}{2}\sum_{i=0}^d\frac{\pi_0}{\pi_i},
	\end{equation}
	where we used that $\sum_{\mbox{\scriptsize $i$ even}} \frac{\pi_0}{\pi_i}=\sum_{\mbox{\scriptsize $i$ odd}} \frac{\pi_0}{\pi_i}$
	(see Fiol \cite{f01}). Thus, $\sum_{i=0}^d\frac{\pi_0}{\pi_i}=\frac{2n}{r}$. Combined with the expressions for the multiplicities, this gives
	\begin{align*}
		\sum_{i=0}^d \frac{\pi_0^2}{m_i\pi_i^2}&=\sum_{\mbox{\scriptsize $i$ even}} \frac{\pi_0}{\pi_i}+\sum_{\mbox{\scriptsize $i$ odd}} \frac{\pi_0}{(r-1)\pi_i}=\sum_{\mbox{\scriptsize $i$ even}}\frac{\pi_0}{\pi_i}\left(1+\frac{1}{r-1}\right)\\
		&=\frac{r}{2(r-1)}\sum_{i=0}^d\frac{\pi_0}{\pi_i}=\frac{n}{r-1}.
	\end{align*}
	Consequently, the spectral excess of $G$ is
	\begin{equation}
		\label{eq:spectral-excess}
		p_d(\theta_0)=n\left(\sum_{i=0}^d \frac{\pi_0^2}{m_i\pi_i^2}\right)^{-1}=r-1=\overline{k}_d,
	\end{equation}
	and the spectral excess theorem implies that $G$ is distance-regular.
	Finally, we use a result of Fiol \cite{f97} stating that a distance-regular graph is $r$-antipodal if and only if the multiplicities are given by the above expressions.
\end{proof}

In fact, to conclude that $G$ is an $r$-antipodal distance-regular graph, some of the above conditions can be relaxed, namely:
\begin{itemize}
	\item
	Since $G$ is assumed to be walk-regular, \eqref{bounds-mult} holds. Hence, to have equalities, we only need to require that the order $n$ is given by \eqref{order}.
	\item
	Alternatively, if we assume equalities in \eqref{bounds-mult}, to infer that $G$ is an $r$-antipodal distance-regular graph, we only need to assume that $G$ is regular (as well as the condition $\overline{k}_d=r-1$).
	%\SZ{This sentence is very confusing}
\end{itemize}

\section{New bounds for $\alpha_k$ using the predistance polynomials}\label{sec:newboundspredistancepolys}
Let $G$ be a graph with spectrum as above. Then we can define the
scalar product
$$
\langle p, q \rangle_G=\frac{1}{n}\tr pq(A)=\frac{1}{n}\sum_{i=0}^d m_i pq(\theta_i)
$$
for $p,q\in \Real_k[x]$. The {\em predistance polynomials} $p_i$, for $i=0,1,\ldots,d$, which were introduced by Fiol and Garriga \cite{fg97} and were used to prove the well-known Spectral Excess Theorem, are a sequence of orthogonal polynomials with respect to the above scalar product, that is,
$$
\langle p_i,p_j\rangle_G=0, \ \mbox{for $i\neq j$},
$$
normalized in such a way that $p_i(\theta_0)=\|p_i\|_G^2$ (see  C\'amara, F\`abrega, Fiol, and Garriga \cite{cffg09} for some applications of these polynomials).

When $G$ is distance-regular, the predistance polynomials become the {\em distance polynomials} that, applied to $A$, give the corresponding distance matrices. In other words,
$p_i(A)=A_i$ for $i=0,\ldots,d$. Thus, in this case, the $k$-power graph $G^k$ has adjacency matrix
$A^{[k]}=q'_k(A)=\sum_{i=1}^k p_i(A)$, hence the (not necessarily distinct) eigenvalues of $G^k$ are
$$
q'_k(\theta_0), q'_k(\theta_1), \ldots, q'_k(\theta_d),
$$
repeated $m_0,m_1,\ldots,m_d$ times respectively.

Since it is known that $q'_k(\theta_0)\ge  q'_k(\theta_i)$ for $i=1,\ldots,d$,
\cite[Coro. 2.4]{cffg09}, we use the bounds  \eqref{bound:cvetkovic}--\eqref{bound:hoffman} on $\alpha_k$ to extend a result by Abiad, Coutinho, Fiol, Nogueira, and Zeijlemaker \cite[Corollary 2.3]{acfnz21}.
\begin{proposition}
	\label{propo-qk}
	$(i)$ Let $G$ be a $k$-partially walk-regular graph on $n$ vertices, with eigenvalues $\lambda_1\ge \cdots \ge \lambda_n$ and predistance polynomials $p_0,p_1,\ldots,p_d$. Let $q'_k=\sum_{i=1}^k p_i$. Then,
	\begin{equation}
		\label{bound:cvetkovic-Gk}
		\alpha_k  \le \min \{|\{i : q'_k(\lambda_i)\ge 0\}| , |\{i : q'_k(\lambda_i)\le 0\}|\}.
	\end{equation}
	$(ii)$ Moreover, if $G$ has different eigenvalues $\theta_0>\theta_1>\cdots>\theta_d$ then,
	\begin{equation}
		\label{bound:hoffman-Gk}
		\alpha_k \leq \frac{n}{1-\frac{q'_k(\theta_0)}{\lambda(q'_k)}}.
	\end{equation}
\end{proposition}
\begin{proof}
	Using the same reasoning as in Dalf\'o, Fiol, and Garriga \cite[Proposition 2.1]{dfg10}, we conclude that $G$ is $k$-partially walk-regular if and only if the matrices $p_i(A)$, for $i=1,\ldots,k$,  have zero diagonals.
	%Since $p_0=1$, we have that $\tr p_i(A)=\langle p_0,p_i\rangle_G=0$ for every $i=1,\ldots,d$.
	%Then, under the hypothesis $\tr q'_k(A)=0$ and \eqref{bound:hoffman-Gk} follows from Corollary \ref{coro:main}$(i)$.
	%Concerning $(ii)$ it was proved that $G$ is walk-regular if and only if the matrices $p_k(A)$, for $k=1,\ldots,d$ have null diagonals.
	Hence, this also holds for the matrix $q'_k(A)$. By considering a suitable  $\alpha_k\times \alpha_k$ principal zero submatrix of $A$ and using interlacing, we then prove $(i)$. Similarly, by taking the appropriate $2\times 2$ quotient matrix, interlacing yields $(ii)$.
	% or in Fiol \cite{fiol20}.
\end{proof}

Next, we study when the bound from Proposition \ref{propo-qk} is tight. Table \ref{table:comparisonalpha2} compares Proposition \ref{propo-qk} to several known upper bounds on the 2-independence number. We limit ourselves to 2-walk-regular graphs that are not distance-regular, since otherwise one can simply use \cite[Corollary 2.3]{acfnz21}. Note that for several graphs, we obtain a better bound than in Abiad, Coutinho, and Fiol \cite{acf19} and again \cite{acfnz21}, for example, for the Gray and Hoffman graphs.
Moreover, if $k=2$, an infinite family for which the bound of Proposition \ref{propo-qk} is tight can be found among the circulant graphs. For any positive integer $n$ and set $S\subseteq \{1,\dots, \lfloor n/2\rfloor\}$, the circulant graph $C_n^{S}$ is the undirected Cayley graph on $\mathbb{Z}_n$ with generating set $S\cup (-S)$. The graph $C_n^{\{s_1, \dots, s_m\}}$ is connected if and only if $\gcd(n,s_1, \dots, s_m) = 1$. In particular, for all $i>j$ such that $\gcd(i,j) = 1$, $C_{2i}^{\{i,j\}}$ is a connected 3-regular graph. If $j$ is odd, these graphs are isomorphic to the M\"obius ladder graphs, and for even $j$ they are prism graphs, also known as circular ladder graphs. In both cases, Proposition \ref{propo-qk} gives a tight bound on the 2-independence number whenever $4\nmid i$. Moreover, this bound is also tight for noncirculant prism graphs if their order is not a multiple of eight. Note that these graphs are all 2-partially walk-regular, but not 2-partially distance-regular.

\begin{table}[t]
	\footnotesize{
		\begin{center}
			\begin{tabular}{|l||r|r|r|r|r|r|}
				\hline
				Graph's name & \cite[Coro. 3.3]{acf19} & $\vartheta_2$ \cite{Lovasz1979OnGraph} & \cite[MILP (20)]{acfnz21} & \cite[Thm.\ 4.2]{acfnz21} & Prop. \ref{propo-qk} & $\alpha_2$ \\
				\hline
				\hline
				Balaban 10-cage & $17$ & $17$ & $19$ & $19$ & $18$ & $17$ \\
				\hline
				Frucht  & $3$ & $3$ & $3$ & $3$ & $3$ & $3$ \\
				\hline
				Meredith  & $14$ & $10$ & $10$ & $10$ & $14$ & $10$ \\
				\hline
				Moebius-Kantor  & $4$ & $4$ & $6$ & $4$ & $4$ & $4$ \\
				\hline
				Bidiakis cube & $3$ & $2$ & $4$ & $3$ & $3$ & $2$ \\
				\hline
				Gray  & $14$ & $11$ & $19$ & $19$ & $13$ & $11$ \\
				\hline
				Nauru  & $6$ & $5$ & $8$ & $8$ & $6$ & $6$ \\
				\hline
				Blanusa First Snark  & $4$ & $4$ & $4$ & $4$ & $4$ & $4$ \\
				\hline
				Blanusa Second Snark  & $4$ & $4$ & $4$ & $4$ & $4$ & $4$ \\
				\hline
				Brinkmann & $4$ & $3$ & $6$ & $6$ & $3$ & $3$ \\
				\hline
				Harborth  & $12$ & $9$ & $13$ & $13$ & $11$ & $10$ \\
				\hline
				Harries & $17$ & $17$ & $18$ & $18$ & $18$ & $17$ \\
				\hline
				Bucky Ball & $16$ & $12$ & $16$ & $16$ & $15$ & $12$ \\
				\hline
				Harries-Wong & $17$ & $17$ & $18$ & $18$ & $18$ & $17$ \\
				\hline
				Robertson & $3$ & $3$ & $5$ & $5$ & $3$ & $3$ \\
				\hline
				Hoffman & $3$ & $2$ & $5$ & $4$ & $2$ & $2$ \\
				\hline
				Holt & $6$ & $3$ & $7$ & $7$ & $4$ & $3$ \\
				\hline
				Szekeres Snark & $12$ & $10$ & $13$ & $13$ & $13$ & $9$ \\
				\hline
				Tietze & $3$ & $3$ & $4$ & $3$ & $3$ & $3$ \\
				\hline
				Double star snark & $7$ & $7$ & $9$ & $9$ & $7$ & $6$ \\
				\hline
				Durer & $3$ & $2$ & $3$ & $3$ & $3$ & $2$ \\
				\hline
				Klein 3-regular  & $13$ & $13$ & $19$ & $18$ & $14$ & $12$ \\
				\hline
				Truncated Tetrahedron & $3$ & $3$ & $4$ & $4$ & $3$ & $3$ \\
				\hline
				Dyck  & $8$ & $8$ & $8$ & $8$ & $8$ & $8$ \\
				\hline
				Tutte  & $11$ & $10$ & $13$ & $13$ & $11$ & $10$ \\
				\hline
				F26A  & $6$ & $6$ & $7$ & $7$ & $6$ & $6$ \\
				\hline
				Watkins Snark & $14$ & $9$ & $13$ & $13$ & $13$ & $9$ \\
				\hline
				Flower Snark & $5$ & $5$ & $7$ & $7$ & $5$ & $5$ \\
				\hline
				Markstroem  & $6$ & $6$ & $7$ & $7$ & $6$ & $6$ \\
				\hline
				Folkman  & $4$ & $3$ & $5$ & $5$ & $3$ & $3$ \\
				\hline
				McGee & $6$ & $5$ & $7$ & $6$ & $6$ & $5$ \\
				\hline
				Franklin & $3$ & $2$ & $4$ & $3$ & $3$ & $2$ \\
				\hline
			\end{tabular}
		\end{center}
		\caption{Comparison between $\alpha_2$ and several of its upper bounds.}
		\label{table:comparisonalpha2}
	}
\end{table}

In the extremal case $k=d-1$, we obtain the following result.

\begin{corollary}
	Let $G$ be a  walk-regular graph on $n$ vertices, with distinct eigenvalues $\theta_0> \cdots > \theta_d$ and predistance polynomial $p_d$. Let $\Lambda(p_d)=\max_{i\in [0,d]}p_d(\theta_i)$. Then,
	%\begin{equation}
	%\label{bound:cvetkovic-Gk}
	%\alpha_k  \le \min \{|\{i : q'_k(\lambda_i)\ge 0\}| , |\{i : q'_k(\lambda_i)\le 0\}|\},
	%\end{equation}
	%Moreover, if $G$ has different eigenvalues $\theta_0>\theta_1>\cdots>\theta_d$ then,
	\begin{equation}
		\label{bound(d-1):hoffman-Gk}
		\alpha_{d-1} \leq \frac{n(1+\Lambda(p_d))}{n+\Lambda(p_d)-p_d(\theta_0)}.
	\end{equation}
\end{corollary}
\begin{proof}
	Notice that, since the Hoffman polynomial is $H=p_0+\cdots+p_d$ and $p_0=1$, we have $q'_{d-1}(x)=H(x)-p_d(x)-1$.
	But $H(\theta_0)=n$ and $H(\theta_i)=0$ for $i=1,\ldots,d$. Then,  $q'_{d-1}(\theta_0)=n-p_d(\theta_0)-1$ and $\lambda(q'_{d-1})=-\Lambda(p_d)-1$. Then \eqref{bound:hoffman-Gk} gives the result.\end{proof}

When $G$ is an $r$-antipodal distance-regular graph, the bound in \eqref{bound(d-1):hoffman-Gk} is tight, since   $\Lambda(p_d)=-p_d(\theta_0)=-r+1$. So, we get $\alpha_{d-1}=r$.

\begin{remark}
	Note that for the regular case, all the bounds for $\alpha_k$ directly yield bounds for the distance chromatic number $\chi_k$, see again Abiad, Coutinho, Fiol, Nogueira, and Zeijlemaker \cite[Section 3]{acfnz21} for details.
\end{remark}

\section*{Acknowledgments}
The research of A. Abiad is partially supported by the FWO grant 1285921N.
The research of C. Dalf\'o and M. A. Fiol has been partially supported by
AGAUR from the Catalan Government under project 2017SGR1087 and by
grant PGC2018-095471-B-I00 funded by MCIN/AEI/10.13039/
501100011033 and “ERDF A way of making Europe", by the European Union. The research of C. Dalf\'o was partially funded by grant PID2020-115442RB-I00 from MCIN/AEI/10.13039/50110 0011033.

%%%%%%%%%%%%%%%%%%%%%%%%%%%%%%%%%%%%%%%%%%%%%%%%%%%%%%%
%\newpage

\end{document}